\documentclass[11,letterpaper,leqno]{article}
\usepackage[letterpaper, hmargin=2cm, vmargin={2cm, 2cm}]{geometry}
\usepackage{myDefs}
\usepackage{cancel}
\usepackage{xcolor}

\begin{document}
\title{On the (Dis)connection Between Growth and Primitive Periodic Points}
\author{Adi Gl\"ucksam and Shira Tanny}
\date{}
\maketitle

\begin{abstract}
{In 1972, Cornalba and Shiffman showed that the number of zeros of an order zero holomorphic function in two or more variables can grow arbitrarily fast.  We generalize this finding to the setting of complex dynamics, establishing that the number of isolated primitive periodic points of an order zero holomorphic function in two or more variables can grow arbitrarily fast as well. This answers a recent question posed by L. Buhovsky, I. Polterovich, L. Polterovich, E. Shelukhin and V. Stojisavljevi{\'c}.}
\end{abstract}


\section{Introduction}\label{sec:intro}
{B\'ezout theorems seek to explore the relation between the growth of a holomorphic function and the size of its zeros set.} The classical B\'ezout theorem states that the number of zeros of $n$ polynomials over $n$ variables is bounded by the product of their degrees. A natural question is whether this theorem extends to general holomorphic functions. Here, the notion of degree is replaced by the notion of growth rate, measured by the asymptotic behavior of its \emph{maximum modulus function}, $M_f(r):=\max_{\abs z\le r}\abs{f(z)}.$

Jensen's formula gives a full description of this relation in complex dimension 1. Formally, assuming that $f(0)\neq 0$, Jensen's formula implies that for every $a>1$ there exists $C=C(f,a)>0$ such that 
$$
\# \bset{\abs z\le r: f(z)=0} \leq C\log M_f(a\cdot r) \qquad \text{for all}\qquad r>0.
$$

In \cite{cornalba1972counterexample} Cornalba and Shifmann showed this phenomenon does not extend to higher dimensions. They constructed a holomorphic function $F:\C^2\rightarrow\C^2$ satisfying $\log M_F(r) = O\bb{\log^2 r}$, 
whose zeros grow arbitrarily fast. Here, and everywhere else in the paper, $O\bb{\log^2 r}$ means that there exists a constant $0<c$ for which $\log M_F(r)\le c\log^2 (r)$. \\

In a recent paper, \cite{buhovsky2023persistent},  L. Buhovsky, I. Polterovich, L. Polterovich, E. Shelukhin and V. Stojisavljevi{\'c} proposed a new way to obtain a transcendental analog to B\'ezout theorem, replacing the traditional zero count with the coarse zero count. Roughly speaking, the \emph{coarse count} of zeros of a holomorphic function in $\C^n$ is the number of connected components on which it is small, that contain a zero. They proved that this coarse count can be bounded in terms of the maximum modulus function, thus establishing an analog to B\'ezout  theorem in the setup of coarse count.

This version of B\'ezout  theorem can be imported into the world of complex dynamics. It implies that the coarse count of $p$-periodic points of holomorphic functions is also bounded by a function of their growth rate. Moreover, the Cornalba-Shifmann example (shifted by the identity map) shows that the honest count of fixed points does not obey such a bound. In \cite{buhovsky2023persistent} they asked whether the construction carried out by Cornalba-Shifmann can be generalized to construct an order zero holomorphic function with arbitrarily many primitive periodic points of higher periods.
\begin{defn}
Let $f:\C^n\rightarrow\C^n$ be a holomorphic function. 
\begin{itemize}
	\item We say $z\in\C^n$ is $p$-\emph{primitive periodic point for $f$} (or $p$-PPP for short) if $f^{\circ j}(z)\neq z$ whenever $1\le j\le p-1$, but $f^{\circ p}(z)=z$.
	\item Denote by $\nu_p(f,r)$ the number of $p$-PPP lying in a ball of radius $r$ centered at the origin, namely
	$$
	\nu_p(f,r):=\#\bb{B(0,r)\cap \bset{z, z \text{ is a p-PPP for }f}}.
	$$
\end{itemize}
\end{defn}
\begin{quest}[{\cite[Question 1.12]{buhovsky2023persistent}}]\label{que:Buh_etal}
	Does there exist a transcendental entire map, $f$, of order 0 (i.e., the modulus $M_f(r)$ grows slower than $e^{r^\varepsilon}$ for every $\varepsilon$) for which $\nu_p(f,r)$ grows arbitrarily fast in $p$ and $r$?
\end{quest}
For a fixed period, $p$, a simple modification of the Cornalba--Shiffman example produces a holomorphic function, $f$, of order zero for which $\nu_p(f,r)$  grows arbitrarily fast in $r$; see Example~\ref{ex:single_period} below. However, producing a holomorphic function with many primitive periodic points of varying prescribed periods is more complicated. Our main result states that this is indeed possible.
\begin{thm}\label{thm:PPP_all_rate}
For every sequence of periods, $\bset{p_n}_{n=1}^\infty$, and for every rate, $\bset{m_n}_{n=1}^\infty$, there exists a holomorphic function $F:\C^2\rightarrow\C^2$ satisfying 
$$
\nu_{p_n}(F,2^n)\ge m_n\quad\text{and}\quad \log M_F(r) = O \bb{\log ^2(r)}.
$$
\end{thm}
The constant in $O(\log^2(r))$ is independent of the sequences. 

Unlike the example by Cornalba and Shiffman, the holomorphic functions constructed in Theorem~\ref{thm:PPP_all_rate} are not completely explicit. In particular, it is not clear how to estimate or even bound from above or below their coarse zero count.

The main idea of the proof is to utilize the simple example constructed in Example~\ref{ex:single_period} with different fixed periods, and then use a `dispatcher' function to choose between different periods. The main tool in the construction is a theorem of H\"ormander (see Theorem~\ref{thm:Hormander} below) which guarantees the existence of solutions to non-homogeneous $\bar\partial$-equations with certain integral bounds. This theorem allows one to construct entire functions with distinct behaviours in different regions. The use of H\"ormander's theorem requires an underlying subharmonic function which is large in the area between distinct regions, but very negative where we want a good approximation. Section~\ref{sec:subharmonic} contains the construction of subharmonic functions that will be used in the application of H\"ormander's theorem in Section~\ref{sec:proof}. In Section~\ref{sec:proof} we construct the required holomorphic functions and prove Theorem~\ref{thm:PPP_all_rate}.

\subsection*{Acknowledgements}
We thank Lev Buhovsky and Leonid Polterovich for useful discussions.

A.G. is grateful for the support of the Golda Meir fellowship.

S.T. was partially supported by a grant from the Institute for Advanced Study School of Mathematics,  the Schmidt Futures program,  a research grant from the Center for New Scientists at the Weizmann Institute of Science, and Alon fellowship.


\section{Preliminary Results on Subharmonic functions}\label{sec:subharmonic}
We commence the discussion with the following lemma whose role is to modify a given subharmonic function assigning $-\infty$ to prescribed points. 
\begin{lem}[The Puncture Lemma]\label{lem:punctures}
Let $u$ be a continuous subharmonic function and, for $N\in\N$ or $N=\infty$, let $\bset{B_k}_{k=1}^N$ be a sequence of pairwise disjoint disks, $B_k=B(z_k,r_k)$. Assume that for every $k$, $\underset{z\in B_k}\inf\; \Delta u>0$. Then there exists a subharmonic function $v$ satisfying
\begin{enumerate}
\item $u(z)=v(z)$ whenever $z\nin\bunion k 1 N B_k$.
\item For every $\delta\in\bb{0,1}$, and $k\in\bset{1,2,\cdots,N}$
$$
 \underset{z\in B(z_k,\delta\cdot r_k)}\max\;v(z)\le \underset{z\in B_k}\max\; u(z)-c\cdot r_k^2\cdot \underset{z\in B_k}\inf\; \Delta u\log\bb{\frac1\delta},
$$
where $c>0$ is some uniform constant.
\end{enumerate}
\end{lem}
See Figure~\ref{fig:Punct} for an illustration of the image of the subharmonic function $v$. 

\begin{figure}[ht]
	\centering
	\includegraphics[scale=0.5]{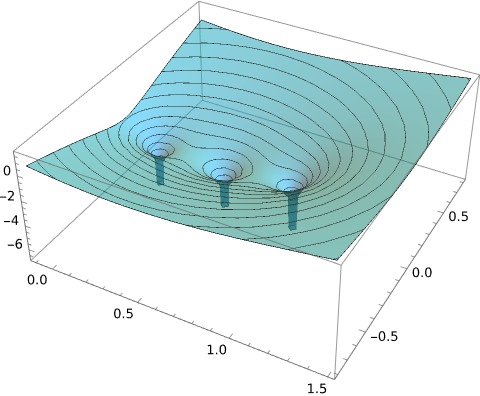}
	\caption{An illustration of the the image of the subharmonic function created by Lemma \ref{lem:punctures}}
	\label{fig:Punct}
\end{figure}

While the proof can be found in \cite[Section 3.1]{glucksam2022colander}, where $\mathcal K_0(t)=\log(t)$ and the function $v$ is any function, we include it here for the reader's convenience.

\begin{proof}
Define the function
$$
v(z)= \begin{cases}
	\bb{P_\D u_k}\bb{\frac{z-z_k}{r_k}}+A_k\log\bb{\frac{\abs{z-z_k}}{r_k}},& z\in B_k\\
	u(z),& \text{otherwise}
	\end{cases}
$$
where $u_k(z)=u\bb{r_k\cdot z+z_k}$ {(and therefore $u(B_k)=u_k(\D)$)}, the constants $A_k>0$ will be chosen momentarily, and $P_\D f$ is Poisson integral of a function $f$ defined by
\[P_\D f (re^{i\theta}):= \frac{1}{2\pi}\int_{-\pi}^\pi P_r(\theta-t)f(e^{it})\ dt,\quad \text{for} \quad P_r(\theta):=\frac{1-r^2}{1-2r\cos \theta+r^2}.\]
Note that $v$ is continuous, since the Poisson integral of any continuous function agrees with it on $\partial \D$. We need to choose the constants $A_k$ such that sub-harmonicity is preserved. We will use the following claim:
\begin{claim}[{\cite[p.24]{evdoridou2023unbounded}}]\label{clm:glueing}
Let $\Omega\subseteq\C$ be a domain, and let $\Omega_1\Subset\Omega$ be a domain, so that $\partial\Omega_1$ is an orientable smooth curve. Every function $f$ which is continuous on $\Omega$ and subharmonic on $\Omega_1\cup \bb{\Omega\setminus\overline{\Omega_1}}$, is subharmonic on $\Omega$ if on $\partial\Omega_1$ it satisfies
$$
\frac{\partial f}{\partial n_1}\le \frac{\partial f}{\partial n_2},
$$
where $n_1$ is the outer normal to $\Omega_1$ along $\partial\Omega_1$ and $n_2$ is the outer normal to $\Omega\setminus\Omega_1$ along $\partial\Omega_1$. 
\end{claim}
We will use this claim with $\Omega:=\C$, and $\Omega_1:=B_k$. In this case, we need to show that along $\partial B_k$
$$
\frac{\partial v}{\partial n_1}\le \frac{\partial v}{\partial n_2}=\frac{\partial u}{\partial n},
$$
and the latter is known. To calculate $\frac{\partial v}{\partial n_1}$ we will use Poisson-Jenssen's formula (see e.g. \cite[Theorem 4.5.1]{ransford1995potential}), which, in this case, boils down to:
\[u_k(z) = P_\D u_k (z) + G_\D(z)\]
where
$$
G_{\D}(z):=\frac1{2\pi}\integrate {\D}{}{g_{\D}(z,w)\Delta u_k(w)}\lambda_2(w), \qquad g_{\D}(z,w):= \log\Big|\frac{z-w}{1-z \bar w}\Big|.
$$
Here, and throughout the rest of the paper, $\lambda_2$ denotes Lebegue's measure in $\C$. Using the above, we see that $v$ is subharmonic if for every $\xi\in \partial B_k$:
\begin{eqnarray*}
\frac{\partial v}{\partial n_1}(\xi)&=&\limit r {1^-}\frac{u(\xi)-P_{\D}(u_k)\bb{\frac{r\bb{\xi-z_k}}{r_k}}}{1-r}+\frac{A_k}{r_k} = \limit r {1^-}\frac{u_k\bb{\frac{\xi-z_k}{r_k}}-u_k\bb{\frac{r\bb{\xi-z_k}}{r_k}}+G_{\D}\bb{\frac{r\bb{\xi-z_k}}{r_k}}}{1-r}+\frac{A_k}{r_k} \\
&=&\frac{\partial \bb{u_k\circ\bb{\frac{z-z_k}{r_k}}}}{\partial n}\bb{\xi}+\limit r {1^-}\frac{G_{\D}\bb{\frac{r\bb{\xi-z_k}}{r_k}}}{1-r}+\frac{A_k}{r_k}\\
&=&\frac{\partial u}{\partial n}\bb{\xi}-\frac{\partial G_{\D}}{\partial r}\bb{\frac{\xi-z_k}{r_k}}\cdot\frac1{r_k}+\frac{A_k}{r_k}\le \frac{\partial u}{\partial n}(\xi)=\frac{\partial v}{\partial n_2}(\xi)
\end{eqnarray*}
using the fact that $G_\D$ vanishes on $\partial \D$ (since $u=P_\D u$ there). Then $\frac{\partial v}{\partial n_1}(\xi)\leq \frac{\partial v}{\partial n_2}(\xi)$ if and only if 
\[
\frac{\partial u}{\partial n}\bb{\xi}-\frac{\partial G_{\D}}{\partial r}\bb{\frac{\xi-z_k}{r_k}}\frac1{r_k}+\frac{A_k}{r_k}\le \frac{\partial v}{\partial n_2}\bb{\xi}=\frac{\partial u}{\partial n}\bb{\xi}\iff A_k\le \underset{\xi\in\partial\D}\min\; \frac{\partial G_{\D}}{\partial r}\bb{\frac{\xi-z_k}{r_k}}
\]
Since the collection $\bset{\frac{\partial g_{\D}\bb{\cdot,w}}{\partial n}(\xi)}_{w\in \D,\xi\in\partial\D}$ is uniformly integrable, we may change the order of the integral and the derivative to obtain that
\begin{eqnarray*}
\frac{\partial G_{\D}}{\partial r}\bb{\frac{\xi-z_k}{r_k}}&=&\frac{\partial}{\partial r}\bb{\frac1{2\pi}\integrate {\D}{}{g_{\D}\bb{\bb{\frac{\xi-z_k}{r_k}},y}\Delta u_k(y)}\lambda_2(y)}=\frac1{2\pi}\integrate {\D}{}{\frac{\partial g_{\D}(\cdot,y)}{\partial r}\bb{\frac{\xi-z_k}{r_k}}\Delta u_k(y)}\lambda_2(y)\\
&\ge&\underset{z\in\D}\inf\; \Delta u_k(z)\cdot \frac1{2\pi}\inf_{\eta\in \D}\integrate {\D}{}{\frac{\partial g_{\D}(\cdot,y)}{\partial r}(\eta)}\lambda_2(y)\ge c\cdot r_k^2\cdot\underset{z\in B_k}\inf\; \Delta u.
\end{eqnarray*}
Note that $c$ is some uniform constant, which does not depend on the function $u$ or the collection $\bset{B_k}$. A direct computation shows that we may choose $c$ to be positive (this follows from the fact that the radial derivative of the Green function $g_\D(\cdot,\eta)$ is everywhere non-negative and does not vanish identically). Setting $A_k:=c\cdot r_k^2\cdot \underset{z\in B_k}\inf\; \Delta u$ guarantees that $v$ is subharmonic.

Finally, for every $\delta\in(0,1)$
\begin{align*}
\underset{\frac{\abs{z-z_k}}{r_k}=\delta}\max\;v(z)=\underset{\frac{\abs{z-z_k}}{r_k}=\delta}\max\; P_\D\bb{u_k}\bb{\frac{z-z_k}{r_k}}+A_k\log\bb{\frac{\abs{z-z_k}}{r_k}}\le \underset{z\in B_k}\max\; u(z)+c\cdot r_k^2\cdot\underset{z\in B_k}\inf\; \Delta u\log(\delta).
\end{align*}
\end{proof}
\section{The proof}\label{sec:proof}
In this section we prove Theorem~\ref{thm:PPP_all_rate}. Extending upon Example~\ref{ex:single_period} below, we need a function that will coordinate between different periods. In Section~\ref{subsec:dispatcher}, we construct a holomorphic function that assumes prescribed values along the sequence $\{2^n\}_{n\in\N_{\ge 2}}$ and whose role is to accommodate for different periods. This part uses the results on subharmonic functions and a theorem of H\"ormander. 
In Section~\ref{subsec:Cornalba_Shiffman}, we modify the construction by Cornalba-Shiffman, \cite{cornalba1972counterexample}, to construct a holomorphic function of two variables with the required number of $p_n$-PPP in every ball $B\bb{0,2^n+1}$, and the required growth rate, thus proving Theorem~\ref{thm:PPP_all_rate}.
\subsection{Constructing the ``dispatcher" functions} \label{subsec:dispatcher}
\begin{lem}\label{lem:dispatcher}
There exists $C>1$ so that for every (finite or infinite) subset $J\subset \N_{n\ge 2}$ and every $M\le (\min J)^2$, there exists an entire function $f:\C\rightarrow\C$ satisfying that
\begin{enumerate}[label=(\roman*)]
\item \label{itm:growth} $M_f(r):=\underset{\abs z=r}\max\abs f\le C\cdot e^{C\log^2(r+1)}$.
\item \label{itm:values} $f\bb{2^n}=\begin{cases}
							M,& n\in J\\
							0,& n\in\N_{n\ge 2}\setminus J
							\end{cases}$.

\end{enumerate}
\end{lem}
\begin{rmk}
	We stress that $C$ is independent of $J$ and $M$.
\end{rmk}

\begin{proof}
As mentioned in the introduction, we use Hörmander's Theorem (Theorem~\ref{thm:Hormander} below) to construct the ``dispatcher" function.  A crucial first step is the construction of an appropriate subharmonic function.

\subsubsection{Step 1: Constructing an appropriate subharmonic function.}
Denote by 
$$
u_0(z):=\begin{cases}
			C\log(\abs z),& \abs z<2\\
			\max\bset{C\log(\abs z),C\log^2(\abs z)},& \abs z\in\bb{2,3}\\
			C\log^2(\abs z),& \text{otherwise}
		\end{cases}
$$
where $C>1$ will be chose momentarily. Note that $u_0$ is subharmonic as a local maximum of subharmonic functions (along $\abs z=2$ the maximum is $C\log\abs z$ and along $\abs z=3$ the maximum is $C\log^2\abs z$). In addition, $u_0$ is a radial subharmonic function and for every $\abs z>3$
$$
\Delta u_0(z)=\Delta u_0(\abs z)={\frac{2C}{\abs z^2}}.
$$
Define the collection of pairwise disjoint disks, $B_k=B\bb{2^k,r_k}$ for $r_k=2^{k-3}$ and $k\in\N_{\ge 2}$, and note that $\bunion k 2\infty B_k\subset\bset{\abs z> 3}$.

We use the Puncture Lemma, Lemma \ref{lem:punctures}, with the function $u_0$ and the collection of disks $\bset{B_k}$ to construct a subharmonic function $u$ satisfying that outside $\bset{B_k}$ we have $u(z)=u_0(z)$ and for every $\delta\in\bb{0,1}$, and $k\in\N_{\ge 2}$
\begin{align}\label{eq:dispatcher}
 \underset{z\in B(2^k,\delta\cdot r_k)}\max\;u(z)&\le \underset{z\in B_k}\max\; u_0(z)-c\cdot r_k^2\cdot\underset{z\in B_k}\inf\; \Delta u_0\log\bb{\frac1\delta}\nonumber\\
 &\le C(k+1)^2-c\cdot 2^{2k-6}\cdot \frac{2C}{2^{2(k+1)}}\log\bb{\frac1\delta} \le 4Ck^2+{4}\log(\delta)
\end{align}
assuming $C\ge\frac{2^9}c$, where $c$ is the constant from the Puncture Lemma, Lemma~\ref{lem:punctures}.

\subsubsection{Step 2: The model map and the holomorphic approximation.}
Let $D_k:=B\bb{2^k, 2^{k-2}}$ and consider any (finite or infinite) subset $J\subset \mathbb{N}_{n\ge 2}$. We define the model map
$$
h(z):=M\cdot \underset{k\in J}\sum\; \indic{D_k}(z)
$$
Let $\chi:\C\rightarrow[0,1]$ be a smooth map satisfying
\begin{enumerate}
\item $\chi(w)=1$ whenever $w\in\bunion k 2\infty B\bb{2^k, \frac{1}{2}2^{k-2}}$.
\item $\chi(w)=0$ whenever $w\in\C\setminus\bunion k 2\infty D_k$.
\item There exists a uniform constant $A>1$ such that
$
\abs{\nabla\chi(w)}\le A\cdot\sumit k 2 \infty 2^{-k}\cdot\indic{D_k\setminus B\bb{2^k, \frac{1}{2}2^{k-2}}}(w).
$
\end{enumerate}
For a construction of such a function, see e.g. \cite[Proposition 2.5]{glucksam2024approximate}.

In order to find a holomorphic map approximating the model map $h$ along the sequence of points $\bset{2^k}_{k\in\N_{n\ge 2}}$ while also providing growth bounds, we use H\"ormander's theorem:
\begin{thm}\label{thm:Hormander}[H\"ormander, {\cite[Theorem 4.2.1]{hormander2007notions}}] 
Let $u:\C\rightarrow\R$ be a subharmonic function. Then, for every locally integrable function $g$ there is a solution $\alpha$ of the equation $\bar\partial \alpha=g$ such that:
\begin{equation}\label{eq:Hormander}
\iint_\C\abs {\alpha(z)}^2\frac{e^{-u(z)}}{\bb{1+\abs z^2}^{2}}d\lambda_2(z)\le\frac12\iint_\C\abs {g(z)}^2e^{-u(z)}d\lambda_2(z),
\end{equation}
provided that the integral in the right hand side is finite.
\end{thm}

We define the function $g(z):=\overline\partial\chi(z)\cdot h(z)=\overline\partial\bb{\chi\cdot h}(z)$, which is supported in $\bigcup_{k\in J} \bb{D_k\setminus B\bb{2^k, \frac{1}{2}2^{k-2}}}$. Note that $u(z)=u_0(z)=C\log^2(|z|)$ on the support of $g$.
To apply H\"ormander's theorem, we bound the following integral:
\begin{align}\label{eq:integral}
&\iint_\C \abs{g(z)}^2\cdot e^{-u(z)}d\lambda_2(z)=\iint_{\bigcup_{k\in J} \bb{D_k\setminus B\bb{2^k, 2^{k-3}}}} \abs{\overline\partial\chi(z)}^2\abs{h(z)}^2\cdot e^{-u(z)}d\lambda_2(z)\nonumber\\
&=\iint_{\bigcup_{k\in J} \bb{D_k\setminus B\bb{2^k, 2^{k-3}}}} \abs{\overline\partial\chi(z)}^2\abs{h(z)}^2\cdot e^{-u_0(z)}d\lambda_2(z)\nonumber\\
&\le A^2\sum_{k\in J}\  \iint_{D_k\setminus B\bb{2^k, 2^{k-3}}} 2^{-2k}M^2 e^{- C(k-1)^2}\le C_1^2\sumit k {1}\infty k^4e^{-C(k-1)^2}:=I<\infty,
\end{align}
since $M\le (\min J)^2$, while if $k\nin J$, the model map satisfies $h|_{D_k}=0$. It is important to note that $I$ is a numerical number which does \underline{NOT} depend on $M$ or on the set $J$.

We apply H\"ormander's theorem with the map $g$ and the subharmonic function $u$, constructed in Step 1, to obtain a function $\alpha:\C\rightarrow\C$ satisfying $\dbar\alpha= g$ and estimate (\ref{eq:Hormander}), and define
$$
f(z)=\chi(z)\cdot h(z)-\alpha(z).
$$
It is not hard to see that $f$ is an entire function, as a solution to the $\bar\partial$ equation.

\subsubsection{Step 3: Bounding the error.}
We will show that for every $k\in\N_{n\ge 2}$ the function, $f$, constructed above, agrees with the model map, $h$, at $2^k$.

Fix $k\in\N_{n\ge 2}$ and note that for every $\delta\in\bb{0,\frac12}$, the function $\chi$ satisfies $\left.\chi\right|_{B\bb{2^k,\delta \cdot 2^{k-2}}}\equiv 1$. Using Cauchy's integral formula for the function $f-h$, which  on $B\bb{2^k,\delta \cdot 2^{k-2}}$ is both holomorphic and coincides with $\alpha$, we obtain that up to a uniform constant,
{\begin{align*}
\abs{f\bb{2^k}-h\bb{2^k}}&=\abs{\alpha\bb{2^k}}\lesssim\frac1{\bb{\delta\cdot 2^{k-2}}^2}\underset{B\bb{2^k,\delta\cdot 2^{k-2}}}\iint |\alpha(w)|d\lambda_2(w)=\frac1{\bb{\delta\cdot 2^{k-2}}^2}\sqrt{\abs{\underset{B\bb{2^k,\delta\cdot 2^{k-2}}}\iint |\alpha(w)|\cdot\overline{\bf 1}d\lambda_2(w)}^2}\\
&\lesssim \frac1{\delta^2\cdot 2^{2k}}\sqrt{\underset{B\bb{2^k,\delta\cdot 2^{k-2}}}\iint \abs{\alpha(w)}^2d\lambda_2(w)\cdot \underset{B\bb{2^k,\delta\cdot 2^{k-2}}}\iint \abs{{\bf 1}(w)}^2d\lambda_2(w)}\\
&\lesssim \frac1{\delta\cdot 2^{k}}\underset{w\in B\bb{2^k,\delta\cdot 2^{k-2}}}\max\; \bb{1+\abs w^2}e^{\frac12u(w)}\sqrt{\iint_\C\abs {\alpha(w)}^2\frac{e^{-u(w)}}{\bb{1+\abs w^2}^{2}}d\lambda_2(w)},
\end{align*}}
where the second inequality follows from Cauchy-Schwartz inequality. Using inequality \eqref{eq:Hormander} from Hormander's Theorem, the above integral can be bounded by the integral of $|g|^2e^{-u}$, which we estimated in \eqref{eq:integral}. Combining these bounds with \eqref{eq:dispatcher}, we obtain 
\begin{align*}
\abs{f\bb{2^k}-h\bb{2^k}}\lesssim &\frac1{\bb{\delta\cdot 2^{k}}}\cdot 2\cdot 2^{2(k+1)}\cdot \exp\bb{ \frac{1}{2}\bb{4Ck^2+4\log(\delta)}}\cdot I\\
\lesssim&\exp\bb{-\log(\delta)+2Ck^2+2\log(\delta)}\le e^{2Ck^2}\cdot{\delta}\underset{\delta\rightarrow0}\longrightarrow 0.
\end{align*}
We conclude that $f\bb{2^k}=h\bb{2^k}$, i.e., property \ref{itm:values} holds.

\subsubsection{Step 4: Bounding the growth.}
To see that property \ref{itm:growth} holds, we use a similar calculation noting that $\abs{h(z)}\le4\log^2\bb{\abs z+1}$, since if $z\nin\underset{k\in J}\bigcup D_k$, then $\abs{h(z)}=0\le4\log^2\bb{\abs z+1}$ and if $z\in\underset{k\in J}\bigcup D_k$ then $\abs{h(z)}=M\le\bb{\underset{k\in J}\min\; k}^2\le 4\log^2\abs z$.

 This time, we apply Cauchy's integral formula to $f$ rather than $\alpha$, {as $\alpha$ is not necessarily holomorphic in these disks. We obtain}
\begin{align*}
\abs{f(z)}&\leq\abs{\chi(z)h(z)-\alpha(z)}=\frac1{\pi}\underset{B\bb{z,1}}\iint |\chi(w)h(w)-\alpha(w)|d\lambda_2(w)\leq \frac{1}{\pi}\underset{B\bb{z,1}}\iint |\chi(w)h(w)|+|\alpha(w)| d\lambda_2(w)\\
&\le \underset{w\in B(z,1)}\sup\abs {h(w)}+\frac1\pi\sqrt{\underset{B\bb{z,1}}\iint |\alpha(w)|^2d\lambda_2(w)}\\
&\lesssim \log^2\bb{\abs z+1}+\underset{w\in B\bb{z,1}}\max\; \bb{1+\abs w^2}e^{\frac12u(w)}\sqrt{\iint_\C\abs {\alpha(w)}^2\frac{e^{-u(w)}}{\bb{1+\abs w^2}^{2}}d\lambda_2(w)}\\
&\lesssim \log^2(\abs z+1)+\bb{2+\abs z^2}e^{\frac C2\log^2(\abs z+1)}\cdot I,
\end{align*}
where we used the fact that $u(w)\leq u_0(w)=C\log^2(|w|)$. Overall, we see that 
$$\abs{f(z)}\lesssim   e^{C\log^2(\abs z+1)},$$
concluding the proof of Lemma~\ref{lem:dispatcher}.
\end{proof}
\subsection{Modifying Cornalba-Shiffman's construction} \label{subsec:Cornalba_Shiffman}
Before constructing an entire function with many primitive periodic points of different periods, let us explain a simple way to construct an entire function with many p-primitive periodic points for a single period, $p$. Roughly speaking, the construction takes (a symmetrized version of) the Cornalba-Shiffman's example and adds to it a standard rotation by $\frac{2\pi}p$, formally defined by $\Theta_p(z)=e^{\frac{2\pi i}p}z$, which rotates the plane by $\frac{2\pi}p$ counter clockwise. Zeros of Cornalba-Shiffman's construction then become $p$-primitive periodic points.
\begin{examp}\label{ex:single_period}
Fix $p\in\N$. Note that while $\Theta_p$ is a $p$-periodic function, its periodic points are not isolated.  Let $F(z,w):\C^2\rightarrow\C^2$ be any entire function of order zero satisfying that in every ball of radius $2^n$, $F$ has $m_n$ isolated zeros that are symmetric under the rotation $\Theta_p$:
$$
(z_j, w_j)\in \C^2,\ j=1,\dots,m_n\quad \text{such that}\quad  F \bb{\Theta_p^{\circ\ell}(z_j), w_j}=0 \quad\text{for all}\quad\ell=0,\dots, k-1.
$$
Here, and everywhere else, $\Theta_p^{\circ\ell}$ is the composition of $\Theta_p$ with itself $\ell$ times. One example of such function is a symmetrized version of the function constructed by Cornalba and Shiffman, which we explicitly synthesize in Lemma~\ref{lem:fancy_CS}. However, any function with isolated symmetric zeroes as above will do. Then
$$
F_p(z,w)=F(z,w)+\bb{\Theta_p(z),w}
$$
is a holomorphic function of order zero satisfying that in every ball of radius $2^n$ it has $m_n$ points which are $p$-PPP.
\end{examp}
We would like to modify this construction, to construct functions with primitive periodic points of varying periods. For this we will use the dispatcher function constructed in the previous subsection. Our main goal for this section is to complete the proof of Theorem~\ref{thm:PPP_all_rate}, which states that for every sequence of periods, $\bset{p_n}_{n=1}^\infty$, and for every rate, $\bset{m_n}_{n=1}^\infty$, there exists a holomorphic function $F:\C^2\rightarrow\C^2$ satisfying 
$$
\nu_{p_n}(F,2^n)\ge m_n\quad\text{and}\quad \log M_F(r) = O ((\log r)^2).
$$
\begin{rmk}\label{rmk:index}
To simplify our notation we will assume our sequences start from $n=2$. As the sequences are arbitrary, this does not limit the generality of the result.
\end{rmk}
\begin{figure}
	\centering
	\includegraphics[scale=0.8]{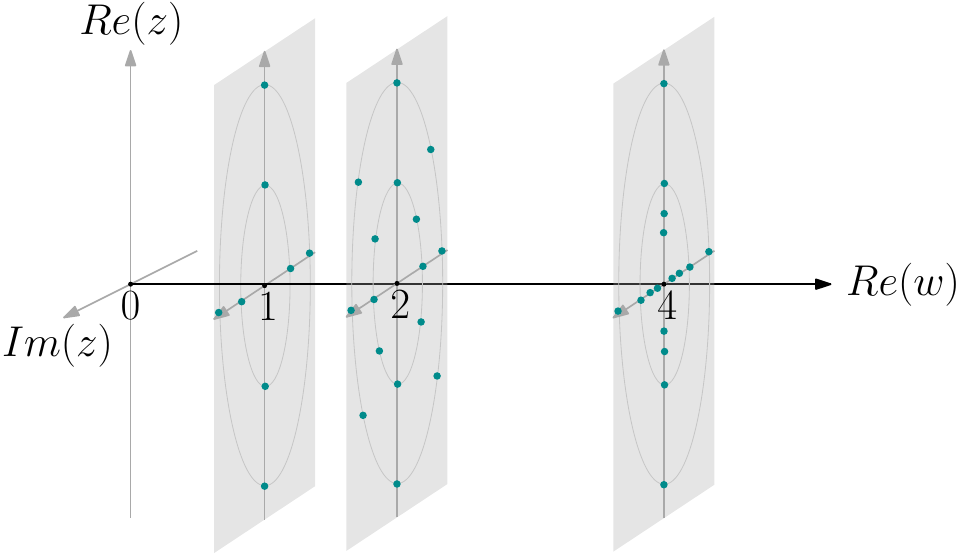}
	\caption{An illustration of the zeroes of a symmetrized Cornalba--Shiffman function, for periods $\{p_n\}=\{4,8,4,...\}$ and rates $\{m_n\}=\{2,2,4,...\}$.}
	\label{fig:Symm_CS}
\end{figure}
We start by constructing a ``symmetrized" version of the Cornalba--Shiffman function, whose zero set is depicted in Figure~\ref{fig:Symm_CS}. 

\begin{lem}\label{lem:fancy_CS}
Given a sequence of periods, $\bset{p_n}$, and a rate, $\bset{m_n}$, there exists a holomorphic function $G=\bb{g_1,g_2}:\C^2\rightarrow\C^2$ of order zero satisfying
\begin{itemize}
\item  The set of isolated zeros of $G$ contains the set
$$
Z:=\bunion n 1\infty Z_n:=\bunion n 1\infty\bset{\bb{\frac{e^{\frac{2\pi i \ell}{p_n}}}{j} ,2^n} :\ 1\le j\le m_n, 0\le\ell\le p_n-1}.
$$
\item $\left.\frac d{dz}g_1\right|_{Z}\neq 0$ and satisfies
$$
\frac d{dz}g_1\bb{\Theta_{p_n}^{\circ\ell}\bb{\frac1j},2^n}=\Theta_{p_n}^{\circ\bb{-\ell}}\bb{\frac d{dz}g_1\bb{\frac1j,2^n}}.
$$
\item $\left.\frac d{dw}g_2\right|_{Z}\neq0$, and $\frac d{dz}g_2\equiv 0$, i.e., $g_2$ does not depend on $z$.
\item There exists a uniform constant $C>1$ (independent of the sequences $\bset{m_n},\bset{p_n}$) such that
$$
\log M_G(R)\le 40\log^2(R+1)+C.
$$
\end{itemize}
\end{lem}
\begin{proof}
Similarly to Carnalba-Shiffman's original construction, let 
$$
Q(w):=\prodit j 1\infty\bb{1-\frac w{2^j}}\quad,\quad\quad Q_n(w):=\underset{j\neq n}\prod\bb{1-\frac w{2^j}}=\frac{Q(w)}{1-\frac w{2^n}},
$$
and for every $n$ consider the following polynomial of degree $(m_n+p_n)$:
$$
P_n(z)=\prodit j 1 {m_n}\prodit \ell 0 {p_n-1}\bb{z-\frac{e^{\frac{2\pi i\ell}{p_n}}}{j}}=\prodit j 1 {m_n}\prodit \ell 0 {p_n-1}\bb{z-z_{j,\ell}}.
$$
where $z_{j,\ell}:=\frac{e^{\frac{2\pi i\ell}{p_n}}}{j}$, are the zeros of $P_n$.

Define the function
$$
G(z,w):=\bb{g_1(z,w),g_2(z,w)}=\bb{\sumit n 1\infty 2^{-\ell_n}\cdot Q_n(w)\cdot P_n(z),\ Q(w)},
$$
where the sequence $\bset{\ell_n}\nearrow\infty$ will be chosen throughout the proof so that $G$ satisfies the assertions of the lemma.

\noindent\textbf{Derivatives and isolation of zeros:} Starting with the requirement on the derivatives, it is clear that $g_2(z,w)=Q(w)$ does not depend on $z$. In addition, the derivative of $g_2$ with respect to $w$ is non-vanishing on $Z$ since $\frac{d}{dw}g_2(z,w)=Q'(w)$ and 
$$
\left.Q'(w)\right|_{w=2^m}=\left.\underset{j}\sum2^{-j} \underset{\nu\neq j}\prod \bb{1-\frac w{2^\nu}}\right|_{w=2^m}=2^{-m}\underset{\nu\neq m}\prod \bb{1-2^{m-\nu}} \neq 0.
$$
The estimate of the derivative of $g_1$ with respect to $z$ requires a computation of the derivatives of $P_n$, using  the symmetry of the set $\bset{e^{\frac{2\pi j}{p_n}}}_{\ell=0}^{p_n-1}$:
\begin{align*}
P_n'\bb{z_{j,\ell}}&=\limit z{z_{j,\ell}}\frac{P_n(z)}{z-z_{j,\ell}}
=\underset{\bb{\nu,k)\neq(j,\ell}}\prod\bb{\frac{e^{\frac{2\pi i\ell}{p_n}}}{j}-\frac{e^{\frac{2\pi ik}{p_n}}}{\nu}}\\
&=\bb{e^{\frac{2\pi i\ell}{p_n}}}^{p_n\cdot m_n-1}\underset{\bb{\nu,k)\neq(j,\ell}}\prod\bb{\frac{1}{j}-\frac{e^{\frac{2\pi i\bb{k-\ell}}{p_n}}}{\nu}}=e^{-\frac{2\pi i\ell}{p_n}}\underset{\bb{\nu,k)\neq(j,\ell}}\prod\bb{\frac{1}{j}-\frac{e^{\frac{2\pi i\bb{k-\ell}}{p_n}}}{\nu}}.
\end{align*}
Let us related the derivatives at $z_{j,\ell}$ and $z_{j,0}$. Shifting the product by $e^{\frac{2\pi i\ell}{p_n}}$ we see that if we fix $\nu\neq j$, the product over $k$ is taken over the entire set $\bset{0,1,\cdots,p_n-1}$ (shifted by $2\pi\ell$). On the other hand, if $\nu=j$ the product that was taken over the set $\bset{0,1,2,\cdots,p_n-1}\setminus\bset\ell$ shifted is the product taken over the set $\bset{1,2,\cdots,p_n-1}$, which is the derivative at  $z_{j,0}$ . We see that
$$
P_n'\bb{z_{j,\ell}}=e^{-\frac{2\pi i\ell}{p_n}}\underset{\bb{\nu,k)\neq(j,0}}\prod\bb{\frac{1}{j}-\frac{e^{\frac{2\pi {k}}{p_n}}}{\nu}}=e^{-\frac{2\pi i\ell}{p_n}}P_n'\bb{z_{j,0}}=\Theta_{p_n}^{\circ\bb{-\ell}}\bb{P_n'\bb{\frac1j}},
$$
implying that
\begin{align*}
\frac d{dz}g_1\bb{\Theta_{p_n}^{\circ\ell}\bb{\frac1j},2^n}&=2^{-\ell_n}\cdot Q(2^n)\cdot P_n'\bb{\Theta_{p_n}^{\circ\ell}\bb{\frac1j}}\\
&=2^{-\ell_n}\cdot Q(2^n)\cdot \Theta_{p_n}^{\circ\bb{-\ell}}\bb{P_n'\bb{\frac1j}}=\Theta_{p_n}^{\circ\bb{-\ell}}\bb{\frac d{dz}g_1\bb{\frac1j,2^n}}.
\end{align*}

\noindent\textbf{Growth bound:} To conclude the proof, we will show that the sequence $\bset{\ell_n}$ can be chosen so that $\log M_G(R)\le 40\log^2(R+1)+C$ for some uniform constant $C$.

We first bound the growth of the functions $Q_n$ independently of $n$. Fix $w\in \C$, and let $k$ be so such that $\abs w\in\left[2^k,2^{k+1}\right)$. Then, 
\begin{align}
\abs{Q_n(w)}&=\abs{\exp\bb{\log\bb{\underset{j\neq n}\prod\bb{1-\frac w{2^j}}}}}\le \exp\bb{\underset{j\neq n}\sum\log\bb{1+\frac{\abs w}{2^j}}}\nonumber\\
&=\exp\bb{\sumit j 1 {k+1}\log\bb{1+\frac{\abs w}{2^j}}}\cdot \exp\bb{\sumit j  {k+2}\infty\log\bb{1+\frac{\abs w}{2^j}}}\le \exp\bb{(k+1)\log\bb{2\abs w}}\cdot \exp\bb{\sumit j  {k+2}\infty\frac{\abs w}{2^j}} \nonumber\\
&\le \exp\bb{2\log^2\abs w}\cdot\exp\bb{2^{-(k+1)}\abs w} \nonumber\le \exp\bb{2\log^2\abs w+1}.
\end{align}
The same computation shows that $|Q(w)|\leq \exp\bb{2\log^2\abs w+1}$ as well, which concludes the growth bound for $g_2$.

Next we bound the growth of $g_1$. The degree of the polynomial $P_n$ is $(m_n+p_n)$, and it is bounded by
$$
\abs{P_n(z)} = \abs{\prodit j 1 {m_n}\prodit \ell 0 {p_n-1}\bb{z-z_{j,\ell}}} = \prodit j 1 {m_n}\prodit \ell 0 {p_n-1}\abs{{z-z_{j,\ell}}} \leq \prodit j 1 {m_n}\prodit \ell 0 {p_n-1}\bb{\abs z+1} = \bb{\abs z+1}^{m_n+p_n}.
$$
For every $k\ge 2$ consider  $(z,w)$ with $\sqrt{\abs z^2+\abs w^2}\in\sbb{2^{k-1},2^k}$. For simplicity of the arguments below, we assume that the sequence $\deg\bb{P_n}=m_n+p_n$ is strictly increasing. Note that this can be achieved by increasing $m_n$, which does not change the statement of the lemma, as $m_n$ is the number of $p_n$-PPP, increasing it will generate more points than required. Define $\mu_k:=\min\bset{n: deg\bb{P_n}\ge 2(k+1)}$. Then, since $\abs z<2^k$,
\begin{align}\label{eq:two_sums}
\abs{g_1(z,w)}&=\abs{\sumit n 1\infty 2^{-\ell_n}Q_n(w)\cdot P_n(z)}\le \underset{n\in\N}\sup\abs{Q_n(w)}\sumit n 1\infty 2^{-\ell_n} \bb{1+\abs z}^{m_n+p_n}\nonumber\\
&\le \exp\bb{2\log^2\abs z+1}\bb{\sumit n 1{\mu_k-1} 2^{\bb{k+1}\bb{m_n+p_n}-\ell_n}+\sumit n {\mu_k}\infty 2^{\bb{k+1}\bb{m_n+p_n}-\ell_n}}.
\end{align}
We shall bound each sum separately. 

To bound the first summand in (\ref{eq:two_sums}), note that we may assume without loss of generality that $\mu_k\le k$ (this happens whenever the sequence $m_n$ grows fast enough) and that $\ell_n> deg(P_n)^2=\bb{m_n+p_n}^2$. Then
\begin{align*}
 \sumit n 1{\mu_k-1} 2^{-\ell_n + (k+1)\cdot (m_n+p_n)} &\le \sumit n 1{\mu_k-1} 2^{-\bb{m_n+p_n}^2  + (k+1)\cdot (m_n+p_n)}\le \mu_k\cdot 2^{2(k+1)^2}\le \exp\bb{2\log(k)+8\cdot k^2}\\
&\le \exp\bb{9 k^2}\le \exp\bb{36\log^2\bb{\sqrt{\abs z^2+\abs w^2}}},
\end{align*}
where in the first inequality we used our choice of $\ell_n>\bb{m_n+p_n}^2$, and in the second, the fact that for $n\leq \mu_k$, $m_n+p_n=deg(P_n)\le 2(k+1)$.

To bound the second sum in (\ref{eq:two_sums}), recall that by the way $\mu_k$ was chosen for every $n\geq \mu_k$, 
$$
m_n+p_n=deg(P_n)\ge deg\bb{P_{\mu_k}}\ge 2\bb{k+1}
$$
implying that
$$
\sumit n {\mu_k}\infty 2^{\bb{k+1}\bb{m_n+p_n}-\ell_n}\leq \sumit n {\mu_k}\infty 2^{-\bb{m_n+p_n}^2\bb{1-\frac{k+1}{m_n+p_n}}}\le \sumit n {\mu_k}\infty 2^{-\frac12\bb{m_n+p_n}^2}<\sumit k 1\infty 2^{-\frac12k^2}<1.
$$


Combining the two estimates together, we see that 
$$
\log M_{g_1}(R)\le 3\log^2(R+1)+36\log^2(R+1)+C\le 39\log^2(R+1)+C,
$$
and thus the desired bound holds for $G$ as well, concluding our proof.
\end{proof}
We will use the lemma above to conclude the construction of a holomorphic map on $\C^2$ with prescribed PPP's for any sequences of periods and any rate.
\begin{proof}[Proof of Theorem \ref{thm:PPP_all_rate}]
In light of remark \ref{rmk:index}, for every $m\ge 2$ we let $D_m:\C\rightarrow\C$ be a dispatcher function created by the Dispatcher Lemma, Lemma \ref{lem:dispatcher}, with $J=\bset{m}$, a singleton, and the constant $M=m^2$. Recall that $D_m$ is a holomorphic function satisfying
$$
D_m(2^n)=\begin{cases}
			m^2&, n=m\\
			0&, \text{otherwise}
		\end{cases}.
$$
We define the function
$$
F(z,w)=G(z,w) + \bb{\sumit m 2\infty \frac1{m^2}D_m(w)\cdot \Theta_{p_m}(z),w},
$$
where $G(z,w)$ is the modification of Cornalba-Shiffman's construction constructed in Lemma~\ref{lem:fancy_CS} above.

We need to show that $F$ has at least $m_n$ isolated $p_n$-PPP in every ball $B(0,2^n+1)$ and that it satisfies the correct growth rate.
As in Lemma~\ref{lem:fancy_CS}, we define the sets
$$
Z_n:=\bset{\bb{\frac{e^{\frac{2\pi i \ell}{p_n}}}{j} ,2^n} :\ 1\le j\le m_n, 0\le\ell\le p_n-1}\subset B(0,2^n+1).
$$
\noindent{\bf Observation 1:} $F|_{Z_n}\bb{z,2^n}=\bb{\Theta_{p_n}(z),2^n}$ and every point in $Z_n$ is a $p_n$-PPP.

Indeed, note that
$$
G|_{Z_n}\equiv 0\quad\quad,\quad\quad\left.\bb{\sumit m 2\infty \frac1{m^2}D_m}\right|_{Z_n}=\frac1{n^2}\cdot D_n(2^n)=1,
$$
implying that
$$
F|_{Z_n}= G|_{Z_n} + \bb{\left.\bb{\sumit m 2\infty \frac1{m^2}D_m(2^n)\cdot \Theta_{p_m}(\cdot)}\right|_{Z_n},2^n}=\bb{\Theta_{p_n}(\cdot),2^n}.
$$
In particular, for every $\bb{z,2^n}\in Z_n$ we have
$$
F^{\circ k}\bb{z,2^n}=F^{\circ (k-1)}\bb{\Theta_{p_n}\bb{z},2^n}=\cdots=\bb{\Theta_{p_n}^{\circ k}\bb{z},2^n}=\bb{z,2^n}\iff k=p_n\cdot \nu\quad,\quad\nu\in\N,
$$
concluding the proof of observation 1.

\noindent{\bf Observation 2:} The points in $Z_n$ are isolated $p_n$-PPP.

Indeed, recall that for every $f:\C\rightarrow\C$, 
$$
\frac d{dz}\bb{f^{\circ {p_n}}}=\prodit k 1 {p_n} \frac d{dz}f\bb{f^{\circ(k-1)}}.
$$
In addition, following Observation 1, $F|_{Z_n}\bb{z,2^n}=\bb{\Theta_{p_n}(z),2^n}$, implying that for every $k$ we have
$$
F^{\circ(k-1)}(z,2^n)=\bb{\Theta_{p_n}^{\circ(k-1)}z,2^n}=\bb{e^{\frac{2\pi i(k-1)}{p_n}}z,2^n}.
$$
We write 
$$
F(z,w)=(f_1(z,w),f_2(z,w))=\bb{g_1(z,w)+\sumit m 1\infty \frac1{m^2}D_m(w)\cdot \Theta_{p_m}(z),g_2(z,w)+w}.
$$
Fix $\bb{\frac1j,2^n}\in Z_n$, we let $z_k=\Theta_{p_n}^{\circ k}\bb{\frac1j}$, then $\bb{z_k,2^n}\in Z_n$ as well and as $\frac d{dz}g_2\equiv 0$,
\begin{align*}
\frac d{dz}\bb{F^{\circ p_n}}\bb{\frac1j,2^n}&=\bb{\prodit k 1 {p_n} \frac d{dz}f_1\bb{F^{\circ(k-1)}\bb{\frac1j,2^n}},\prodit k 1 {p_n} \frac d{dz}f_2\bb{F^{\circ(k-1)}\bb{\frac1j,2^n}}}\\
&=\bb{\prodit k 1 {p_n} \bb{\frac d{dz}g_1\bb{z_{k-1},2^n}+\sumit m 1\infty \frac1{m^2}D_m(2^n)\cdot \frac d{dz}\Theta_{p_m}\bb{z_{k-1}}},\prodit k 1 {p_n} \frac d{dz}g_2\bb{z_{k-1},2^n}}\\
&=\bb{\prodit k 1 {p_n}  \bb{\frac d{dz}g_1\bb{z_{k-1},2^n}+e^{\frac{2\pi i}{p_n}}},0},\\
\frac d{dw}\bb{F^{\circ p_n}}\bb{\frac1j,2^n}&= \bb{\prodit k 1 {p_n} \bb{\frac d{dw}g_1\bb{z_{k-1},2^n}+\sumit m 1\infty \frac1{m^2}\frac d{dw}D_m(2^n)z_{k-1}},1+\frac d{dw}g_2\bb{z_{k-1},2^n}}.
\end{align*}

If $dF^{\circ{p_n}} - \operatorname{Id}$ is degenerate at $\bb{\frac1j,2^n}\in Z_n$, then
\begin{align*}
0=det\bb{\frac d{dz} F^{\circ p_n}-Id,\frac d{dw} F^{\circ p_n}-Id}&=\bb{\prodit k 1 {p_n}  \bb{\frac d{dz}g_1\bb{z_{k-1},2^n}+e^{\frac{2\pi i}{p_n}}}-1}\bb{1+\frac d{dw}g_2\bb{z_{k-1},2^n}-1}\\
&=\frac d{dw}g_2\bb{z_{k-1},2^n}\bb{\prodit k 1 {p_n}  \bb{\frac d{dz}g_1\bb{z_{k-1},2^n}+e^{\frac{2\pi i}{p_n}}}-1}\\
&\iff  \prodit k 1 {p_n}  \bb{\frac d{dz}g_1\bb{z_{k-1},2^n}+e^{\frac{2\pi i}{p_n}}}=1\text{ or } \frac d{dw}g_2\bb{z_{k-1},2^n}=0.
\end{align*}
However, on one hand, $\left.\frac d{dw}g_2\right|_{Z}\neq0$ by Lemma~\ref{lem:fancy_CS}. On the other hand, following the same lemma, if we denote by
$$
a:=\frac d{dz}g_1\bb{\frac1j,2^n},
$$
then for every $k\ge 2$,
$$
\frac d{dz}g_1\bb{z_{k-1},2^n}=\frac d{dz}g_1\bb{\Theta_{p_n}^{\circ\bb{k-1}}\bb{\frac1j},2^n}=\Theta_{p_n}^{\circ\bb{-\bb{k-1}}}\bb{\frac d{dz}g_1\bb{\frac1j,2^n}}=e^{-\frac{2\pi i\bb{k-1}}{p_n}}\cdot a.
$$
We conclude that
$$
\prodit k 1 {p_n}\bb{\frac d{dz}g_1\bb{z_{k-1},2^n}+e^{\frac{2\pi i}{p_n}}}=\prodit k 1 {p_n}\bb{e^{-\frac{2\pi i k}{p_n}}\cdot a+1}=\prodit k 0 {p_n-1}\bb{1+a\cdot e^{\frac{2\pi ik}{p_n}}}
$$
implying that
$$
\abs{ \prodit k 1 {p_n}  \bb{\frac d{dz}g_1\bb{z_{k-1},2^n}+e^{\frac{2\pi i}{p_n}}}-1}= \abs{\prodit k 0 {p_n-1}\bb{1+a\cdot e^{\frac{2\pi ik}{p_n}}}-1}.
$$
To estimate the latter, note that if $\varphi:=e^{\frac{2\pi i}{p_n}}$, then $\bb{z^{p_n}-1}=\prodit k 0 {p_n-1}\bb{z-\varphi^k}$ as the both are polynomials of degree $p_n$ with the same $p_n$ roots (of unity) and leading coefficient 1.
We derive from that,
$$
\prodit k 0 {p_n-1}\bb{1+a\cdot e^{\frac{2\pi ik}{p_n}}}=\prodit k 0 {p_n-1}\bb{1-a\cdot \varphi^k}=a^{p_n}\prodit k 0 {p_n-1}\bb{\frac1a- \varphi^k}=a^{p_n}\bb{\frac1{a^{p_n}}-1}=1-a^{p_n}
$$
and therefore
$$
\abs{\prodit k 1 {p_n}\bb{\frac d{dz}g_1\bb{z_{k-1},2^n}+e^{\frac{2\pi i}{p_n}}}-1}=\abs{1-\bb{\frac d{dz}g_1\bb{\frac1j,2^n}}^{p_n}-1}=\abs{\frac d{dz}g_1\bb{\frac1j,2^n}}^{p_n}>0
$$
as $\frac d{dz}g_1\bb{\frac1j,2^n}\neq 0$. We conclude that $dF^{\circ{p_n}} - \operatorname{Id}$ is non-degenerate on $Z$, and the points in $Z_n$ are isolated $p_n$-PPP.

In fact, the same calculation shows that for every $\ell\in\bset{1,\cdots,p_n-1}$ the point $\bb{\Theta_{p_n}^{\circ\ell}\bb{\frac1j},2^n}=\bb{\frac{e^{\frac{2\pi i\ell}{p_n}}}j,2^n}\in Z_n$ is isolated $p_n$-PPP as well.

\noindent{\bf Observation 3:} Growth bound- $\log M_F(R)\le C\log^2(R+1)+C$.

We use the bounds presented in Lemmas~\ref{lem:dispatcher} and \ref{lem:fancy_CS}. For any $R>10$
\begin{align*}
\log M_F(R)&\le \log M_G(R)+\underset{m\in\N}\sup \log M_{D_m}(R)+2\log(R+1)\\
&\le 40\log^2(R+1)+C+C\log^2(R+1)\le 2C\log^2(R+1)+C
\end{align*}
for some uniform constant $C>40$. In other words, $F$ is of order zero.
\end{proof}


\bibliographystyle{plain}
\bibliography{refs}

\bigskip
\noindent A.G.: Einstein Institute of Mathematics, Edmond J. Safra Campus, The Hebrew University of Jerusalem, Givat Ram. Jerusalem, 9190401, Israel
\newline{\tt https://orcid.org/0000-0002-6957-9431}
\newline{\tt adi.glucksam@mail.huji.ac.il}

\bigskip
\noindent S.T.: Department of Mathematics, Weizmann Institute of Science 234 Herzl St. PO Box 26. Rehovot ,7610001, Israel
\newline {\tt https://orcid.org/0000-0001-6389-6922}
\newline{\tt tanny.shira@gmail.com} 

\end{document}